\documentclass[12pt]{amsart}
\usepackage{amsmath,amssymb,amsthm,upref,graphicx,mathrsfs,esint}
\usepackage{color}
\usepackage[
 colorlinks=true,
 linkcolor=blue,
 citecolor=blue,
 urlcolor=blue,
]{hyperref}

\numberwithin{equation}{section}
\allowdisplaybreaks

\newtheorem{theorem}{Theorem}[section]
\newtheorem{prop}[theorem]{Proposition}
\newtheorem{lemma}[theorem]{Lemma}
\newtheorem{cor}[theorem]{Corollary}
\theoremstyle{definition}
\newtheorem{definition}[theorem]{Definition}

\newtheorem{remark}[theorem]{Remark}

\newcommand{\R}{\mathbb{R}}

\newcommand{\wt}[1]{\widetilde{#1}}

\def\Xint#1{\mathchoice
  {\XXint\displaystyle\textstyle{#1}}%
  {\XXint\textstyle\scriptstyle{#1}}%
  {\XXint\scriptstyle\scriptscriptstyle{#1}}%
  {\XXint\scriptscriptstyle\scriptscriptstyle{#1}}%
  \!\int}
\def\XXint#1#2#3{{\setbox0=\hbox{$#1{#2#3}{\int}$}
    \vcenter{\hbox{$#2#3$}}\kern-.5\wd0}}

\def\avgint{\Xint-}%
\DeclareMathOperator*{\esssup}{ess\,sup}  

\let \O=\Omega
\let \ve=\varepsilon

\begin{document}


  \title[Weighted inequalities in SHT]{Calder\'on--Zygmund operators and commutators in Spaces of Homogeneous Type: weighted inequalities}

  \date{\today}

  \author[T.C. Anderson]{Theresa C. Anderson}
  \address{Theresa C. Anderson\\
           Department of Mathematics, Brown University, Providence, RI 02912, USA}
  \email{tcanderson@math.brown.edu}

  \author[W. Dami\'an]{Wendol\'in Dami\'an}
  \address{Wendol\'in Dami\'an\\
  Departamento de An\'alisis Matem\'atico, Facultad de Matem\'aticas,
  Universidad de Sevilla, 41080 Sevilla, Spain}
  \email{wdamian@us.es}

  \thanks{The first author is supported by an NSF graduate fellowship. The second author is supported by Junta de Andaluc\'ia (Grant No. P09-FQM-4745) and the Spanish Ministry of Science and Innovation (MTM2012-30748).}

%
%
  \begin{abstract}
     The recent proof of the sharp weighted bound for Calder\'on-Zygmund operators has led to much investigation in sharp mixed bounds for operators and commutators, that is, a sharp weighted bound that is a product of at least two different $A_p$ weight constants.  The reason why these are sought after is that the product will be strictly smaller than the original one-constant bound.  We prove a variety of these bounds in spaces of homogeneous type, using the new techniques of Lerner, for both operators and commutators.
  \end{abstract}
%
%
  \maketitle
%


  \section{Introduction}\label{Sect1}

In the last decades, harmonic analysts have paid much attention to the area of weighted inequalities for singular integrals.  Since Muckenhoupt introduced in \cite{Mu} the $A_p$ classes in the 1970's to answer the necessary and sufficient conditions for the boundedness of the maximal function on weighted $L^p$ spaces, that is for $1<p<\infty$
 \begin{equation}
 \|Mf\|_{L^p(w)}\leq C\|f\|_{L^p(w)}
 \end{equation}
 if and only if $w\in A_p$,

 \begin{equation}
   [w]_{A_p} = \sup_Q\avgint_Qw\left( \avgint_Qw^{1-p'}\right) ^{p-1} <\infty,
 \end{equation}

 many have sought a deeper understanding of the constants present in such weighted bounds.

\par The first result in this area was due to Buckley, who asserted the precise dependence of the $A_p$ constant in \cite{Bu} was
\begin{equation} \label{buckley}
\|M\|_{L^p(w)}\leq c_p[w]_{A_p}^{\frac{1}{p-1}}.
 \end{equation}
Since then, much work was put into solving the so-called $A_2$ conjecture, which stated that the constant in the corresponding weighted norm inequality for Calder\'on--Zygmund singular integrals depended linearly on the $A_2$ constant.  From there, one could extrapolate to get the $A_p$ dependence.  This open question was recently solved by Hyt\"onen in \cite{Hyt1}. See also \cite{Hyt2} for a survey about the history of the conjecture and \cite{L} for  a simpler proof.

\par  After the solution of the $A_2$ conjecture, an improvement of this result was obtained in \cite{HP}. This new result can be better understood if we first consider the case of Buckley's estimate \eqref{buckley} for the maximal function. For the case $p=2$, the maximal estimate is given by
\begin{equation*}
\|M\|_{L^2(w)}\leq c_n\, [w]_{A_2}.
\end{equation*}
The idea is to replace a portion of the $A_2$ constant by another smaller constant defined
in terms of the $A_{\infty}$ constant given by the functional
\begin{equation}\label{FW}
[w]_{A_{\infty}} = \sup_{Q} \frac{1}{w(Q)}\int_Q M(w\chi_Q).
\end{equation}
To be more precise, the improvement of Buckley's theorem  is the following ``$A_2-A_{\infty}$" estimate
\begin{equation*}
\|M\|_{L^2(w)}\leq c_n\, [w]_{A_2}^{1/2}\, [w^{-1}]_{A_\infty}^{1/2}
\end{equation*}
which can be found in \cite{HP} along with its $L^p$ counterpart.  The $A_{\infty}$ constant  as given by \eqref{FW} was originally introduced  by Fujii in \cite{Fu} and rediscovered later by Wilson in \cite{Wil1}. This definition is more suitable than the more classical condition due to Hrus\v{c}\v{e}v \cite{Hru}, which is defined by the expression
\begin{equation*}
[w]_{A_{\infty}}^{H}= \sup_{Q} \left(\frac{1}{|Q|}\int_Q w(t)dt\right)exp\left(\frac{1}{|Q|}\int_Q \log w(t)^{-1}dt\right),
\end{equation*}
since
$$[w]_{A_\infty}\leq c_n\,[w]^{H}_{A_\infty}, $$
as it was observed in \cite{HP}.  In fact it is shown in the same paper with explicit examples  that  $[w]_{A_\infty}$ is much smaller (actually exponentially smaller) than $[w]^{H}_{A_\infty}$.

Considering the case of singular integrals, the mixed sharp $A_2-A_{\infty}$ result below obtained in \cite{HP}
improves the $A_2$ theorem:
\begin{equation*}
\|T\|_{L^2(w)}\leq c_n\, [w]_{A_2}^{1/2}\, ( [w^{-1}]_{A_\infty}^{1/2} + [w]_{A_\infty}^{1/2}).
\end{equation*}
This is the right estimate when compared with $M$ since this reflects the property that $T$ is (essentially) self-adjoint.

In this paper we follow this idea of replacing a portion of the $A_p$ constant by
the $A_\infty$ constant for the problems considered in \cite{LOP1} and improved in \cite{LOP2}, within the context of spaces of homogeneous type.

To do this we will prove first the following proposition which follows essentially from \cite{LOP2}, but in this stated form can be found in \cite{Per1}.

\begin{theorem}
  Let $T$ be a Calder\'on--Zygmund operator and let $1<p<\infty$. Then for any weight $w$ and $r>1$,

  \begin{equation}\label{LinearGrowthIneq}
    ||Tf||_{L^p(w)} \leq C p p' (r')^{\frac{1}{p'}} ||f||_{L^p(M_r w)}.
  \end{equation}

  \end{theorem}

This mixed theorem leads to sharp mixed $A_1-A_{\infty}$ weighted bounds for Calder\'on--Zygmund operators using simple properties of $A_p$ weights:

       \begin{equation}\label{LinearGrowthCor1}
        ||Tf||_{L^p(w)} \leq C p p' [w]_{A_{\infty}}^{1/p'} ||f||_{L^p(Mw)},
     \end{equation}
  and if $w\in A_1$,

     \begin{equation}\label{LinearGrowthCor2}
      ||Tf||_{L^p(w)} \leq C p p' [w]_{A_{\infty}}^{1/p'} [w]_{A_1}^{1/p} ||f||_{L^p(w)},
     \end{equation}
  where $C$ is a dimensional constant that also depends on $T$.

\par  Recent results relating to the simple proof of the $A_2$ conjecture have allowed us not only to simplify and streamline the proof of this theorem, but to also extend it to the versatile spaces of homogeneous type.  This extension is due to the fact that every Calder\'on--Zygmund operator is bounded from above by a supremum of sparse operators in spaces of homogeneous type.  A sparse operator is an averaging operator over a sparse family of cubes of the following form,
\[T^S f = \sum_{Q\in S}(\avgint_Q f)\chi_Q,\] where a sparse family $S$ is defined by the property that if $Q\in S$,
\[\mu( \bigcup_{Q'\subsetneq Q, \, Q'\in S})\leq \frac{\mu(Q)}{2}.\]  Then the decomposition \[\|Tf\|_X\leq \sup_{S,D}\|T^Sf\|_X,\] where $X$ is a Banach function space, and $D$ a dyadic grid, was shown to hold in $\R^n$ by Lerner and was key in his proof of many weighted conjectures \cite{L}, as it was in the extension of these results to spaces of homogeneous type \cite{ACM} and \cite{AV}.

\par A similar approach can also be done with sharp weak bounds.  We have the following sharp bound of Lerner, Ombosi and P\'erez in \cite{LOP2},

 \[\|T\|_{L^p(w)}\leq Cpp'[w]_{A_1},\]

 which led to the following endpoint estimates in \cite{HP}. Namely,

  \begin{enumerate}
    \item If $w\in A_{\infty}$

      \begin{equation*}
        ||Tf||_{L^{1,\infty}(w)} \leq C \log{(e+[w]_{A_{\infty}})} ||f||_{L^1(M_r w)}.
      \end{equation*}

    \item If $w\in A_1$

      \begin{equation*}
        ||Tf||_{L^{1,\infty}(w)} \leq C  [w]_{A_1} \log{(e+[w]_{A_{\infty}})} ||f||_{L^1(w)}.
      \end{equation*}

  \end{enumerate}

  Again, we can update this proof by applying recent estimates in sharp weighted theory as well as extend to spaces of homogeneous type.  We will employ a new sharp reverse H\"older inequality for spaces of homogeneous type (see Lemma \ref{SRHI}).  By a Calder\'on-Zygmund decomposition and an estimate of the maximal function that comes about from the use of sparse operators, we can arrive at the result.

\par We are also able to extend the sharp $L^p$ and $A_q$ bounds from Duoandikoetxea to spaces of homogeneous type in \cite{Duo} as it is shown in the following result.

\begin{cor}
Let $T$ be an operator such that
  \begin{equation*}
    ||Tf||_{L^p(w)} \leq C N([w]_{A_1}) ||f||_{L^p(w)}
  \end{equation*}
  for all weights $w\in A_1$ and all $1<p<\infty$, with $C$ independent of $w$. Then we have

  \begin{equation}\label{Ap_Duo_estimate}
    ||Tf||_{L^p(w)} \leq C N([w]_{A_q}) ||f||_{L^p(w)}
  \end{equation}

  for all $w\in A_q$ and $1\leq q < p < \infty$, with $C$ independent of $w$,
  \end{cor}

  Finally, we show sharp weighted bounds for commutators of Calder\'on--Zygmund operators with functions in $BMO$ and their iterates. These questions have been considered before in \cite{CPP} and lately improved  in  \cite{HP} but our bounds are new in the context of spaces of homogeneous type. Moreover, we adapt the necessary lemmas to our situation, such as the sharp John-Nirenberg inequality.

  \par The organization of this paper will be as follows. In Section \ref{Sect2} we give some background and definitions which will help us to prove our main results, listed in Section \ref{Sect3}. Finally, Section \ref{Sect4} contains all the proofs as well as some remarks.


  \section{Preliminaries}\label{Sect2}

      \subsection{Spaces of homogeneous type}

We will be working on spaces of homogeneous type, which generalize the Euclidean situation of $\R^n$ with Lebesgue measure.  Other examples of spaces of homogeneous type include $C^{\infty}$ compact Riemannian manifolds, graphs of Lipschitz functions and Cantor sets with Hausdorff measure.  These and more examples are described in \cite{C}; some applications of these spaces can be found in \cite{MM,Y}.

\begin{definition}

A space of homogeneous type is an ordered triple $(X, \rho, \mu)$ where $X$ is a set, $\rho$ is a quasimetric, that is:

\begin{enumerate}
  \item $\rho(x,y)=0$ if and only if $x=y$.
  \item $\rho(x,y)=\rho(y,x)$ for all $x,y \in X$.
  \item $\rho(x,z)\leq \kappa (\rho(x,y)+\rho(y,z))$, for all $x,y,z\in X$.
\end{enumerate}

for some constant $\kappa>0$ (quasimetric constant), and the positive measure $\mu$ is doubling, that is

$$0< \mu(B(x_0,2r))\leq D_{\mu} \mu(B(x_0,r))<\infty,$$

for some constant $D_{\mu}$ (doubling constant).

\end{definition}

We will say that a constant is absolute if it only depends on the space $(X,\rho,\mu)$. Particularly, $\kappa$ and $D_{\mu}$ appearing in the above definition are absolute constants.

Thankfully, many basic constructions and tools for classical harmonic analysis still exist in some form in spaces of homogeneous type, such as certain covering lemmas.  We will especially use the Lebesgue Differentiation Theorem, very recently shown to hold in spaces of homogeneous type in \cite{ACM} where the usual standard assumptions have been removed.

We also rely on a dyadic grid decomposition, and use the new construction of \cite{HK}.

\begin{theorem}[\cite{HK}]\label{dyadic} There exists a family of sets  $D=\cup_{k\in \mathbb{Z}}D_k$, called a dyadic decomposition of $X$, constants $0<C,\epsilon <\infty$, and a corresponding family of points $\lbrace x_c(Q)\rbrace_{Q\in D}$ such that:

  \begin{enumerate}
    \item $X=\bigcup_{Q\in D_k}Q$, for all $k\in \mathbb{Z}$.
    \item If $Q_1\cap Q_2 \neq \emptyset$, then $Q_1\subseteq Q_2$ or $Q_2\subseteq Q_1$.
    \item For every $Q\in D_k$ there exists at least one child cube $Q_c\in D_{k-1}$ such that $Q_c\subseteq Q$.
    \item For every $Q\in D_k$ there exists exactly one parent cube $\hat{Q}\in D_{k+1}$ such that $Q\subseteq \hat{Q}$.
    \item If $Q_2$ is a child of $Q_1$ then  $\mu(Q_2)\geq \epsilon \mu(Q_1)$.
    \item $B(x_c(Q), \delta^k)\subset Q \subset B(x_c(Q), C\delta^k)$.
  \end{enumerate}

\end{theorem}
  We will refer to the last property as the \emph{sandwich property}.

  A weight $w$ is a nonnegative locally integrable function on $(X,\mu)$ that takes values in $(0,\infty)$ almost everywhere. For any $1<p<\infty$ we define the $A_p$ constant of the weight $w$ on the space of homogeneous type $X$ as follows

  \begin{equation}\label{Ap_mu}
    [w]_{A_p} := \sup_Q \left(\frac{1}{\mu(Q)} \int_Q w(x)d\mu\right) \left(\frac{1}{\mu(Q)} \int_Q w(x)^{1-p'} d\mu \right)^{p-1}.
  \end{equation}

  Here we can take $Q$ to be either in the family of dyadic cubes or balls, since the concept of a non-dyadic cube is not defined in spaces of homogeneous type. Note that the $A_p$ constant is comparable when considering suprema over families of dyadic cubes or balls by using the sandwich property of dyadic cubes and the doubling property of the measure $\mu$.

  It is crucial to note that we will take this constant with respect to cubes (by which we always will mean dyadic cubes), because the $A_{\infty}$ constant below is not comparable in the same way, even in the classical Euclidean case.

  \par We define the Fujii--Wilson $A_{\infty}$ constant in a space of homogeneous type as follows:

  \begin{equation*}
     [w]_{A_{\infty}} = \sup_{Q} \frac{1}{w(Q)}\int_Q M(w\chi_Q)d\mu.
  \end{equation*}

  While this $A_{\infty}$ constant is comparable using dyadic cubes or balls, the constant of comparison depends on the measure $w$ (which is doubling since $w\in A_{\infty}$).  Hence to achieve sharp bounds in spaces of homogeneous type we cannot simply switch between these constants defined with respect to cubes or balls since we introduce a $w$-dependent factor.  This is reflected in the difference between the sharp reverse H\"older inequalities involving this $A_{\infty}$ constant: if defined with respect to cubes we get a sharp reverse H\"older (see Lemma \ref{SRHI}), but if defined with respect to balls we only get a sharp weak version (see \cite{HPR}).

      \subsection{Calder\'on--Zygmund operators and commutators}

Next we recall some definitions related to Calder\'on--Zygmund operators and their commutators in the homogeneous setting.

\begin{definition}\label{calderon}

We say that $K:X \times X \setminus\lbrace{x=y\rbrace}\to \R$ is a Calder\'on--Zygmund kernel if there exist $\eta>0$ and $C<\infty$ such that for all $x_0\neq y\in X$ and $x\in X$ it satisfies the decay condition:

\begin{equation}\label{decay}
|K(x_0,y)|\leq \frac{C}{\mu(B(x_0,\rho(x_0,y)))}
\end{equation}

and the smoothness conditions for $\rho(x_0,x)\leq \eta\rho(x_0,y)$:

\begin{equation}\label{smoothness1}
|K(x,y)-K(x_0,y)|\leq\left(\frac{\rho(x,x_0)}{\rho(x_0,y)}\right)^{\eta}\frac{1}{\mu(B(x_0,\rho(x_0,y)))},
\end{equation}

and

\begin{equation}\label{smoothness2}
|K(y,x)-K(y,x_0)|\leq\left(\frac{\rho(x,x_0)}{\rho(x_0,y)}\right)^{\eta}\frac{1}{\mu(B(x_0,\rho(x_0,y)))}.
\end{equation}

\end{definition}

\begin{definition}
Let $T$ be a singular integral operator associated to Calder\'on--Zygmund kernel $K$. If in addition $T$ is bounded on $L^2$, we say that $T$ is a Calder\'on--Zygmund operator.
\end{definition}

Next let us recall some boundedness properties of the Calder\'on--Zygmund operators.

\begin{theorem}[\cite{CW2}] Let $T$ be a Calder\'on--Zygmund operator on a space of homogeneous type.  Then $T$ is bounded from $L^1$ to $L^{1,\infty}$.
\end{theorem}

We will be invoking sparse operators to bound our Calder\'{o}n-Zygmund operators using the formula originally due to Lerner \cite{L}.  Before we discuss this, we need to define a sparse family on a dyadic grid $D=\cup_k D_k$ as in Theorem \ref{dyadic}.

  \begin{definition}
   A sparse family $S = \cup_k S_k$, $S_k\in D_k$ on $D$ is a collection of dyadic cubes such that for $Q',Q \in S$, \[\mu(\bigcup_{Q'\subsetneq Q}Q')\leq \frac{\mu(Q)}{2}.\]

  \end{definition}

  \begin{definition}
  Given a sparse family $S$, we define a sparse operator as follows
     \[T^S(f) = \sum_{Q\in S}(\avgint_Qf)\cdot \chi_Q.\]
  \end{definition}

  We also introduce the decomposition of Lerner, proved in the homogeneous setting in \cite{ACM}.

  \begin{theorem}
  For any Calder\'{o}n-Zygmund operator $T$ on a space of homogeneous type $X$, we have that

  $$\|Tf\|_Y \leq C \sup_{D,S}\|T^Sf\|_Y,$$

  where $D$ is a dyadic grid, $C$ only depends on the operator and the space $X$, and $Y$ is any Banach function space.

  \end{theorem}

  And finally, we introduce the definitions of a $BMO$ function and the iterated commutators of Calder\'on--Zygmund operators with functions in $BMO$ in spaces of homogeneous type.

  \begin{definition}
    For a locally integrable function $b:X\longrightarrow\R$ we define

      \[ ||b||_{BMO} = \sup_Q \frac{1}{\mu(Q)} \int_Q |b(y)-b_Q|d\mu(y) < \infty,\]
    where the supremum is taken over all dyadic cubes in $X$, and

    \[b_Q = \frac{1}{\mu(Q)} \int_Q b(y)d\mu(y).\]

  \end{definition}

  \begin{definition}
    Given a Calder\'on--Zygmund operator $T$ with kernel $K$ and a function $b$ in $BMO$, we define the k-th order commutator with $b$, for an integer $k\geq 0$, as follows

    $$ T_b^k(f)(x) = \int_X (b(x)-b(y))^k K(x,y)f(y) d\mu(y).$$
  \end{definition}

  In the particular case when $k=1$, $T_b^1$ is the classic commutator and we will denote it by $T_b$.

  \par Throughout this paper, $X$ will denote a space of homogeneous type equipped with a quasimetric $\rho$ with quasimetric constant $\kappa$ and a positive doubling measure $\mu$ with doubling constant $D_{\mu}$. We will denote by $C$ a positive constant independent of the weight constant which may change from a line to other.


   \section{Main results}\label{Sect3}

Our goal is to prove the following results.


\begin{theorem}\label{LinearGrowthThm}

  Let $T$ be a Calder\'on--Zygmund operator and let $1<p<\infty$. Then for any weight $w$ and $r>1$,

  \begin{equation}\label{LinearGrowthIneq}
    ||Tf||_{L^p(w)} \leq C p p' (r')^{\frac{1}{p'}} ||f||_{L^p(M_r w)},
  \end{equation}

  where $C$ is an absolute constant that also depends on $T$.

\end{theorem}

From the previous theorem we obtain the following estimates as immediate corollaries.


\begin{cor}\label{LinearGrowthCor}

  Let $T$ be a Calder\'on--Zygmund operator and let $1<p<\infty$. Then if $w\in A_{\infty}$ we obtain

       \begin{equation}\label{LinearGrowthCor1}
        ||Tf||_{L^p(w)} \leq C p p' [w]_{A_{\infty}}^{1/p'} ||f||_{L^p(Mw)},
     \end{equation}
  and if $w\in A_1$,

     \begin{equation}\label{LinearGrowthCor2}
      ||Tf||_{L^p(w)} \leq C p p' [w]_{A_{\infty}}^{1/p'} [w]_{A_1}^{1/p} ||f||_{L^p(w)},
     \end{equation}
  where $C$ is an absolute constant that also depends on $T$.

\end{cor}


As an application of \eqref{LinearGrowthIneq} we obtain the following endpoint estimate.

\begin{theorem}\label{WeakEstimateThm}
    Let $T$ be a Calder\'on--Zygmund operator. Then for any weight $w$ and $r>1$,

    \begin{equation}\label{WeakEstimateIneq}
      ||Tf||_{L^{1,\infty}(w)} \leq C \log{(e+r')} ||f||_{L^1(M_r w)},
    \end{equation}

    where $C$ is an absolute constant that also depends on $T$.

\end{theorem}


Additionally we get the following estimates as corollaries of the above result choosing $r$ as the sharp exponent in the reverse H\"older inequality for weights in the $A_{\infty}$ class in the setting of spaces of homogeneous type (see Lemma \ref{SRHI}) and taking into account that $r'\approx [w]_{A_{\infty}}$.

\begin{cor}\label{WeakEstimateCor}
  Let $T$ be a Calder\'on--Zygmund operator. Then

  \begin{enumerate}
    \item If $w\in A_{\infty}$

      \begin{equation*}
        ||Tf||_{L^{1,\infty}(w)} \leq C \log{(e+[w]_{A_{\infty}})} ||f||_{L^1(M_r w)}.
      \end{equation*}

    \item If $w\in A_1$

      \begin{equation*}
        ||Tf||_{L^{1,\infty}(w)} \leq C  [w]_{A_1} \log{(e+[w]_{A_{\infty}})} ||f||_{L^1(w)}.
      \end{equation*}

  \end{enumerate}
  In both cases $C$ is an absolute constant that also depends on $T$.

\end{cor}


A natural extension of our earlier listed results for Calder\'on--Zygmund operators is a version of an extrapolation theorem from Duoandikoetxea \cite{Duo}, for spaces of homogeneous type.  Here we get sharp bounds involving the $A_p$ weight constant through an initial $A_{p_0}$ boundedness assumption.

\begin{theorem}\label{Duo_ExtrapolationThm}
  Assume that for some family of pairs of nonnegative functions $(f,g)$, for some $p_0\in [1,\infty)$, and for all $w\in A_{p_0}$ we have

  \begin{equation*}
    \left(\int_{X} g^{p_0} w \right)^{1/p_0} \leq C N([w]_{A_{p_0}}) \left(\int_{X} f^{p_0} w \right)^{1/p_0},
  \end{equation*}
  where $N$ is an increasing function and the constant $C$ does not depend on $w$. Then for all $1<p<\infty$ and all $w\in A_p$ we have

  \begin{equation*}
    \left(\int_{X} g^p w \right)^{1/p} \leq C K(w) \left(\int_{X} f^p w \right)^{1/p},
  \end{equation*}
  where

  \begin{equation*}
    K(w)= \left\{
            \begin{array}{ll}
              N([w]_{A_p}(2\|M\|_{L^p(w)})^{p_0-p}, & \hbox{if $p<p_0$;} \\
              N([w]_{A_p}^{\frac{p_0-1}{p-1}}(2\|M\|_{L^{p'}(w^{1-p'})})^{\frac{p-p_0}{p-1}}, & \hbox{if $p>p_0$.}
            \end{array}
          \right.
  \end{equation*}
  In particular, $K(w)\leq C_1 N(C_2 [w]_{A_p}^{\max{1,\frac{p_0-1}{p-1}}})$, for $w\in A_p$ where $C_2$ is an absolute constant.
\end{theorem}

As application of the last result we obtain an estimate for $L^p(w)$ norms with $A_q$ weights for $q<p$.

\begin{cor}\label{DuoThm}
  Let $T$ be an operator such that

  \begin{equation*}
    ||Tf||_{L^p(w)} \leq C N([w]_{A_1}) ||f||_{L^p(w)}
  \end{equation*}
  for all weights $w\in A_1$ and all $1<p<\infty$, with $C$ independent of $w$. Then we have

  \begin{equation}\label{Ap_Duo_estimate}
    ||Tf||_{L^p(w)} \leq C N([w]_{A_q}) ||f||_{L^p(w)}
  \end{equation}
  for all $w\in A_q$ and $1\leq q < p < \infty$, with $C$ independent of $w$. In particular, \eqref{Ap_Duo_estimate} holds with $N(t)=t$ if $T$ is a Calder\'on--Zygmund operator.

\end{cor}

  We also prove the following bound for a Calder\'on--Zygmund operator that is useful to get sharp $A_2-A_{\infty}$ bounds for the commutators in spaces of homogeneous type.

  \begin{theorem}\label{mixedsharp}
    Let $T$ be a Calder\'on--Zygmund operator and $w\in A_2$. Then the following sharp weighted bound in an space of homogeneous type holds:
      \[\|T\|_{L^2(w)}\leq C[w]_{A_2}^{1/2}([w]_{A_{\infty}}+[\sigma]_{A_{\infty}})^{1/2}.\]
  \end{theorem}

  And finally as a corollary of the previous result and using a precise version of the John--Nirenberg inequality proved in Section \ref{Subsect44}, we prove the following generalized sharp weighted bound for the k-th iterate commutator of a Calder\'on--Zygmund operator.

  \begin{cor}\label{commutator_thm}
    Let $T$ be a Calder\'on--Zygmund operator defined on a space of homogeneous type and $b\in BMO$. Then

     \begin{equation}\label{comm2}
      \|T^k_b(f)\|_{L^2(w)} \leq C [w]_{A_2}^{1/2} ([w]_{A_{\infty}}+[\sigma]_{A_{\infty}})^{k+1/2} \|b\|_{BMO}\|f\|_{L^2(w)}.
     \end{equation}
    where $C$ is an absolute constant. In particular, for the classical commutator we get the following estimate
     \begin{equation}\label{comm1}
       \|T_b(f)\|_{L^2(w)} \leq C [w]_{A_2}^{1/2} ([w]_{A_{\infty}}+[\sigma]_{A_{\infty}})^{3/2} \|b\|_{BMO} \|f\|_{L^2(w)}.
     \end{equation}
  \end{cor}

\begin{remark}
The optimality of the exponents in these results follow from the corresponding results in $\R^n$ which were obtained by building specific examples of weights for each operator. However, a new approach to derive the optimality of the exponents  without building explicit examples can be found in \cite{LPR}.
\end{remark}


   \section{Proofs}\label{Sect4}

      \subsection{Proofs of Theorem \ref{LinearGrowthThm} and Corollary \ref{LinearGrowthCor}}\label{Subsect41}

First, we will prove the next inequality of Coifman--Fefferman type.

\begin{prop}\label{Coifman_Fefferman_prop}
  Let $T$ be a Calder\'on--Zygmund operator and let $1<p<\infty$. If $w\in A_p$ then

  \begin{equation}\label{Coifman_Fefferman_ineq}
    \int_X |Tf(x)| w(x) d\mu(x) \leq C [w]_{A_p} \int_X Mf(x) w(x) d\mu(x),
  \end{equation}
  where $C$ is an absolute constant that depends also on $T$.
\end{prop}

Before proving Proposition \ref{Coifman_Fefferman_prop} we need to recall the following lemma that will allow us to obtain the precise constant in \eqref{Coifman_Fefferman_ineq} and that can be found in \cite[Ex. 9.2.5]{Graf2} as well as in \cite[p. 388]{GCRF} in the context of two weights.

\begin{lemma}\label{Aq_ineq_Lemma}

  Let $\mu$ be a positive doubling measure and $1<p<\infty$. If $w\in A_p$ then

  \begin{equation}\label{Aq_inequality}
    \left(\frac{\mu(A)}{\mu(Q)}\right)^p \leq [w]_{A_p} \frac{w(A)}{w(Q)},
  \end{equation}
  where $A\subset Q$ is a $\mu$-measurable set and $Q$ is a cube.

\end{lemma}

\begin{proof}[Proof of Proposition \ref{Coifman_Fefferman_prop}]

  We have that

  \begin{equation*}
    \int_X |Tf(x)|w(x) d\mu(x) \leq C_{X,T}  \sup_{S,D}\int_X \left|\sum_{Q\in S}\left( \avgint_Q f(x)\right) \chi_Q(x) \right|w(x)d\mu(x)
  \end{equation*}
  by the formula of Lerner proved in the homogeneous setting in \cite{ACM}. Using \eqref{Aq_inequality}, we obtain that

  \begin{equation*}\begin{split}
    \int_X |Tf(x)|w(x)d\mu(x) &\leq  C_{X,T} \sup_{S,D}\sum_{Q\in S}\left( \avgint_Q|f(x)|\right) w(Q) \\
                    &\leq C_{X,T}\sup_{S,D} [w]_{A_p}\sum_{Q\in S}\left( \avgint_Q f(x)\right) w(E(Q)) \\
                    &\leq C_{X,T} [w]_{A_p} \sup_{D,S}\sum_{Q\in S}\int_{E(Q)} Mf(x) w(x) d\mu(x),
  \end{split}\end{equation*}

  Finally, since the family $E(Q)$ is disjoint, we can bound the above by

  \begin{equation*}\begin{split}
    \int_X|Tf(x)|w(x)d\mu(x) &\leq C [w]_{A_p} \sup_{D,S}\int_X Mf(x) w(x) d\mu(x) \\
    &\leq C [w]_{A_p} \int_X Mf(x) w(x) d\mu(x),
  \end{split}\end{equation*}
  where $C$ is an absolute constant that depends also on $T$, proving \eqref{Coifman_Fefferman_ineq} as wanted.

\end{proof}

  Now we prove the following lemma.

  \begin{lemma}\label{TrickyLemma}

  Let $w$ be any weight and let $1\leq p,r < \infty$. Then there is a constant $C=C_{X,T}$ such that

  \begin{equation*}
    ||Tf||_{L^p((M_rw)^{1-p})} \leq C p ||Mf||_{L^p((M_rw)^{1-p})}.
  \end{equation*}

  \end{lemma}

  The proof of this lemma is based in a variation of the Rubio de Francia algorithm that could be found in \cite{Per1}.

  \begin{proof}[Proof of Lemma \ref{TrickyLemma}]

    We want to prove

    \begin{equation*}
      \left|\left| \frac{Tf}{M_r w}\right|\right|_{L^p(M_r w)} \leq C p \left|\left| \frac{Mf}{M_r w}\right|\right|_{L^p(M_r w)}.
    \end{equation*}

    By duality we have

    \begin{equation*}
      \left|\left| \frac{Tf(x)}{M_r w(x)}\right|\right|_{L^p(M_r w)} = \left| \int_X Tf(x) h(x) d\mu(x) \right| \leq \int_X |Tf(x)| |h(x)| d\mu(x),
    \end{equation*}

    for some $h$ such that $||h||_{L^{p'}(M_r w)}=1$. By a variation of Rubio de Francia's algorithm addapted to spaces of homogeneous type (see \cite[Lemma 4.4]{Per1}) with $s=p'$ and $v=M_r w$ there exists an operator $R$ such that

    \begin{enumerate}
      \item $0\leq h\leq R(h)$.
      \item $||R(h)||_{L^{p'}(M_r w)} \leq 2 C_{X,p} ||h||_{L^{p'}(M_r w)}$.
      \item $[R(h)(M_r w)^{1/p'}]_{A_1}\leq C_X p$.
    \end{enumerate}

    Let us recall two facts: First, if two weights $w_1,w_2\in A_1$, then $w=w_1w_2^{1-p}\in A_p$ and $[w]_{A_p}\leq [w_1]_{A_1}[w_2]_{A_1}^{p-1}$. Second, by the Coifman-Rochberg theorem in spaces of homogeneous type \cite[Prop. 5.32]{CMP}, if $r>1$ then $(Mf)^{1/r}\in A_1$ and $[(Mf)^{1/r}]_{A_1}\leq C_X r'$. Combining these facts and $(3)$ we obtain

    \begin{equation*}\begin{split}
      [R(h)]_{A_3} &=    [R(h) (M_r w)^{1/p'} ((M_r w)^{1/2p'})^{-2}]_{A_3} \\
               &\leq [R(h) (M_r w)^{1/p'}]_{A_1} [(M_r w)^{1/2p'}]_{A_1}^2 \\
               &\leq C_X p((2p'r)')^2 \leq C_X p,
    \end{split}\end{equation*}
since $(2p'r)'<2$.

    Thus, by Proposition \ref{Coifman_Fefferman_ineq} and using $(1)$ and $(2)$, we obtain

    \begin{equation*}\begin{split}
      \int_X |Tf(x)| h(x) d\mu(x) &\leq \int_X |Tf(x)| R(h)(x) d\mu(x) \\
                       &\leq C [R(h)]_{A_3} \int_X M(f)(x)R(h)(x) d\mu(x) \\
                       &=    C [R(h)]_{A_3} \int_X M(f)(x)R(h)(x)(M_r w(x))^{-1} M_r w(x) d\mu(x) \\
                       &\leq C [R(h)]_{A_3} \left|\left| \frac{M(f)}{M_r w}\right|\right|_{L^p(M_r w)} ||R(h)||_{L^{p'}(M_r w)} \\
                       &\leq C [R(h)]_{A_3} \left|\left| \frac{M(f)}{M_r w}\right|\right|_{L^p(M_r w)} ||h||_{L^{p'}(M_r w)} \\
                       &\leq C p \left|\left| \frac{M(f)}{M_r w}\right|\right|_{L^p(M_r w)},
    \end{split}\end{equation*}

    and we are done.

\end{proof}

And finally using Lemma \ref{TrickyLemma} applied to $T^*$ we can prove Theorem \ref{LinearGrowthThm}.

\begin{proof}[Proof of Theorem \ref{LinearGrowthThm}]

First we are going to prove the following inequality

  \begin{equation}\label{MainLemmaIneq}
    ||Tf||_{L^p(w)} \leq C p p' (r')^{1-\frac{1}{pr}} ||f||_{L^p(M_rw)},
  \end{equation}
from which follows \eqref{LinearGrowthIneq}. Indeed,  it is clear that

\begin{equation*}
  1-\frac{1}{pr} = 1 -\frac{1}{p} + \frac{1}{p} -\frac{1}{pr} = \frac{1}{p'} - \frac{1}{pr'},
\end{equation*}
and since $t^{1/t}\leq 2$ when $t\geq 1$, it follows that

\begin{equation*}
(r')^{1-\frac{1}{pr}} = (r')^{\frac{1}{p'} - \frac{1}{pr'}} \leq 2^{-1/p} (r')^{\frac{1}{p'}}.
\end{equation*}

Next consider the dual estimate of \eqref{MainLemmaIneq}, namely

\begin{equation*}
  ||T^* f ||_{L^{p'}((M_r w)^{1-p'})} \leq C pp' (r')^{1-\frac{1}{pr}} ||f||_{L^{p'}(w^{1-p'})},
\end{equation*}

where $T^*$ is the adjoint operator of $T$. Then, since $T$ is also a Calder\'on--Zygmund operator we are under assumptions of Lemma \ref{TrickyLemma} for $T^*$, we get

\begin{equation*}
  \left|\left| \frac{T^* f}{M_r w}\right|\right|_{L^{p'}(M_r w)} \leq C p' \left|\left| \frac{Mf}{M_r w} \right|\right|_{L^{p'}(M_r w)}.
\end{equation*}

Using H\"older's inequality with exponent $pr$ we have

\begin{equation*}\begin{split}
  \frac{1}{|Q|} \int_Q f &w^{-1/p} w^{1/p} d\mu \\
  &\leq \left(\frac{1}{|Q|} \int_Q w^r d\mu \right)^{1/pr} \left(\frac{1}{|Q|} \int_Q (f w^{-1/p})^{(pr)'}d\mu(x)\right)^{1/(pr)'},
\end{split}\end{equation*}

and hence,

\begin{equation*}
  M(f)^{p'} \leq (M_r w)^{\frac{p'}{p}} M((fw^{-1/p})^{(pr)'})^{p'/(pr)'}.
\end{equation*}

From this and the unweighted maximal theorem in spaces of homogeneous type that can be easily obtained from the proof in \cite{Graf1}, changing the dimensional constant for a geometric one, we obtain

\begin{equation*}\begin{split}
  \left( \int_X \frac{M(f)^{p'}}{(M_r w)^{p'-1}} d\mu \right) ^{1/p'} &\leq  \left( \int_X M((fw^{-1/p})^{(pr)'})^{p'/(pr)'} d\mu \right) ^{1/p'}\\
                                              &\leq  C\left(\frac{p'}{p'-(pr)'}\right)^{1/(pr)'} \left( \int_X f^{p'}w^{1-p'} d\mu \right) ^{1/p'}\\
                                              &=C\left(\frac{p'}{p'-(pr)'}\right)^{1/(pr)'} \left|\left|\frac{f}{w}\right|\right|_{L^{p'}(w)} \\
                                              &= C \left(\frac{rp-1}{r-1}\right)^{1-1/pr}\left|\left|\frac{f}{w}\right|\right|_{L^{p'}(w)} \\
                                              & \leq C p \left(\frac{r}{r-1}\right)^{1-1/pr} \left|\left|\frac{f}{w}\right|\right|_{L^{p'}(w)},
\end{split}\end{equation*}

proving \eqref{MainLemmaIneq} and consequently \eqref{LinearGrowthIneq}.

\end{proof}

\begin{proof}[Proof of Corollary \ref{LinearGrowthCor}]

The proofs of \eqref{LinearGrowthCor1} and \eqref{LinearGrowthCor2} are immediate. In the first case, the estimate is derived by applying the sharp reverse H\"older inequality for weights in the $A_\infty$ class proved in Lemma \ref{SRHI} to \eqref{LinearGrowthIneq} and using the fact that $r'\approx [w]_{A_{\infty}}$. The latter is a direct consequence of \eqref{LinearGrowthCor1} since $w\in A_1$.
\end{proof}

      \subsection{Proof of Theorem \ref{WeakEstimateThm}}\label{Subsect42}

First we establish a lemma which follows similar ideas of \cite[Ch. 4, Lemma 3.3]{GCRF}, that we will need for the proof of Theorem \ref{WeakEstimateThm}.

\begin{lemma}\label{GCRF_lemma}
  Let $T$ be a Calder\'on--Zygmund operator. If $w$ is a weight and $a\in L^1(w)$ supported in a cube $Q$ with $\int_Q a(y)d\mu(y)=0$. Then, if
  we set $\tilde{Q}=LQ$ for a large $L>\eta>0$, the following inequality holds

  \begin{equation}\label{GCRF_lemma}
    \int_{X\setminus \tilde{Q}} |T(a)(x)|w(x)d\mu(x) \leq C \int_X |a(x)| Mw(x) d\mu(x),
  \end{equation}
  with $C$ an absolute constant depending on the kernel $K$.
\end{lemma}

\begin{proof}
  Fix $y_0\in X$ and assume for simplicity that $Q=B(y_0,R)$, with $R>0$. Now making use of the cancellation property of $a$, we obtain

  \begin{equation*}\begin{split}
    \int_{X\setminus \tilde{Q}} |T(a)(x)| w(x) d\mu(x) &= \int_{X\setminus \tilde{Q}} \left| \int_Q K(x,y) a(y) d\mu(y) \right| w(x) d\mu(x) \\
    &\!\leq \!\int_Q \!\int_{X \setminus \tilde{Q}}\! |K\!(x,y)\!-\!K\!(x,y_0)| w(x) d\mu(x) |a(y)| d\mu(y) \\
    &\leq \int_Q I(y) |a(y)| d\mu(y).
  \end{split}\end{equation*}

  Then we only need to prove that $I$ is bounded by $C Mw(y)$ where $C=C_{X,K}$ is an absolute constant depending on the kernel $K$. For every $y\in Q$, using the smoothness property of $K$ in the second variable since $\rho(y,y_0)\leq\eta \rho(x,y_0)$, we obtain

  \begin{equation*}\begin{split}
    I(y) &= \int_{X\setminus\tilde{Q}} |K(x,y)-K(x,y_0)| w(x) d\mu(x) \\
         &= \int_{X\setminus\tilde{Q}} \left(\frac{\rho(y,y_0)}{\rho(x,y_0)}\right)^{\eta}\frac{1}{\mu(B(y_0,\rho(x,y_0)))} w(x) d\mu(x) \\
         &= \sum_{l=1}^{\infty} \int_{2^l Q \setminus 2^{l-1}Q} \left(\frac{\rho(y,y_0)}{\rho(x,y_0)}\right)^{\eta}\frac{1}{\mu(B(y_0,\rho(x,y_0)))} w(x) d\mu(x)
  \end{split}\end{equation*}

  \begin{equation*}\begin{split}
         &\leq \sum_{l=1}^{\infty} \int_{2^l Q} \frac{2^{\eta}}{2^{l\eta}} \frac{\mu(B(y_0,2^{l}\delta^k))}{\mu(B(y_0,2^{l-1}\delta^k))}\frac{1}{\mu(2^l Q)} w(x) d\mu(x) \\
         &\leq D_{X,K} \sum_{l=1}^{\infty} \frac{1}{2^{l\eta}} \frac{1}{\mu(2^l Q)}\int_{2^l Q} w(x) d\mu(x) \\
         &\leq D_{X,K} Mw(y).
  \end{split}\end{equation*}
   Above we have used the fact that $\rho(y,y_0)<R$ and since $\rho(x,y_0)>LR$, there exists $l>1$, so that $2^{l-1}R<\rho(x,y_0)<2^l R$  . Then we have shown that \eqref{GCRF_lemma} holds.

\end{proof}
  \begin{proof}[Proof of Theorem \ref{WeakEstimateThm}]

The proof of Theorem \ref{WeakEstimateThm} is based on several ingredients that we will mention as we need them. We follow the proof of Theorem 1.6 in \cite{Per2}.  We claim that the following inequality holds: for any $1<p,r<\infty$
  \begin{equation}\begin{split}\label{weak_CLAIM}
    ||Tf||_{L^{1,\infty}(w)} &= \lambda \sup_{\lambda>0} w(\{ y\in X : |Tf(y)|>\lambda \}) \\
                             &\leq C (p')^{p} (r')^{p-1} ||f||_{L^{1}(M_r w)},
  \end{split}\end{equation}
  and this claim implies \eqref{WeakEstimateIneq}. Indeed, it suffices to fix $r>1$ and choose $p = 1+\frac{1}{\log{r'}}$.  We then obtain

  \begin{equation*}\begin{split}
    ||Tf||_{L^{1,\infty}(w)} &\leq C (p')^{p-1} (1+\log{r'}) (r')^{p-1} ||f||_{L^{1}(w)} \\
                             &\leq C \log{(e+r')} ||f||_{L^{1}(M_r w)}.
  \end{split}\end{equation*}
  since $p'=1+\log{r'}$, $(p')^{p-1} = (1+\log(r'))^{1/{\log(r')}}\leq e$, $(r')^{p-1} = (r')^{1/{\log(r')}}=e$ and $1+\log{r'}= \log{(er')}\leq 2\log(e+r')$.

We now prove \eqref{weak_CLAIM}. By the classical Calder\'on--Zygmund decomposition of a function $f\in \mathcal{C}^{\infty}_0(X)$ at a level $\lambda$ (see, for example, \cite{ACM}), we obtain a family of non-overlapping dyadic cubes $\{Q_j\}$ satisfying

  \begin{equation*}
    \lambda < \frac{1}{|Q_j|} \int_{Q_j} |f(x)| d\mu(x) \leq C_X \lambda.
  \end{equation*}

  Let $\O = \cup_{j}\O_j$ and let us denote $\wt{Q}_j= 2 Q_j$ and $\wt{\O}= \cup_j \wt{Q}_j$. Using the notation $f_Q=\frac{1}{|Q|}\int_Q f(x)d\mu(x)$, we write $f=g+b$ where $g(x)=\sum_j f_{Q_j}\chi_{Q_j}(x) + f(x)\chi_{\O^c}$ and $b=\sum_j b_j$ with $b_j(x)=(f(x)-f_{Q_j})\chi_{Q_j}(x)$. We have

  \begin{equation*}\begin{split}
    w(\{ y\in X : |Tf(y)|>\lambda \}) &\leq  w(\wt{\O}) + w(\{ y\in(\wt{\O})^c : |Tg(y)|>\lambda/2 \}) \\
    & + w(\{ y\in(\wt{\O})^c : |Tb(y)|>\lambda/2 \}) \equiv I + II + III.
  \end{split}\end{equation*}

Now
  \begin{equation*}\begin{split}
    I = w(\wt{\O}) &\leq C \sum_j \frac{w(\wt{Q}_j)}{\mu(\wt{Q}_j)} \mu(Q_j) \leq \frac{C}{\lambda} \sum_j \frac{w(\wt{Q}_j)}{\mu(\wt{Q}_j)} \int_{Q_j}|f(x)| d\mu(x) \\
    &\leq \frac{C}{\lambda} \sum_j \int_{Q_j}|f(x)|Mw(x) d\mu(x) \leq \frac{C}{\lambda} \int_X |f(x)| Mw(x)d\mu(x) \\
    & \leq \frac{C}{\lambda} \int_X |f(x)| M_r w(x)d\mu(x).
  \end{split}\end{equation*}

  The second term is estimated as follows using Chebyshev's inequality and \eqref{LinearGrowthIneq}.  For each $p>1$ we get

  \begin{equation*}\begin{split}
    II &= w(\{ x\in(\wt{\O})^c: |T(g)(x)|>\lambda/2 \}) \\
       &\leq C (p')^p (r')^{p/p'} \frac{1}{\lambda^p} \int_X |g|^p M_r(w\chi_{(\wt{\O})^c}) d\mu(x) \\
       &\leq C (p')^p (r')^{p/p'} \frac{1}{\lambda} \int_X |g| M_r(w\chi_{(\wt{\O})^c}) d\mu(x) \\
       &\leq C (p')^p (r')^{p/p'} \frac{1}{\lambda} \left( \int_{X\setminus\O} |g|M_r(w\chi_{(\wt{\O})^c}) d\mu(x) + \int_{\O} |g|M_r(w\chi_{(\wt{\O})^c}) d\mu(x) \right) \\
       &= C (p')^p (r')^{p/p'} \frac{1}{\lambda} (II_1 + II_2).
  \end{split}\end{equation*}

  It is clear that

  \begin{equation*}
    II_1 = \int_{X\setminus\O} |g| M_r(w\chi_{(\wt{\O})^c}) d\mu(x) \leq \int_X |f| M_r w d\mu(x).
  \end{equation*}

  Next, we estimate $II_2$ as follows

  \begin{equation*}\begin{split}
    II_2 &= \int_{\O} |g| M_r(w\chi_{(\wt{\O})^c})(x) d\mu(x) \\
      &\leq \sum_j \int_{Q_j} |f_{Q_j}| M_r(w\chi_{(\wt{\O})^c})(x) d\mu(x) \\
      &\leq \sum_j \int_{Q_j} \frac{1}{\mu(Q_j)}\int_{Q_j} |f(y)|d\mu(y)  M_r(w\chi_{(\wt{\O})^c})(x) d\mu(x)\\
      &\leq C \sum_j \int_{Q_j} |f(y)|d\mu(y) \frac{1}{\mu(Q_j)} \inf_{Q_j} M_r(w\chi_{(\wt{\O})^c})\mu(Q_j) \\
      &\leq C \sum_j \int_{Q_j} |f(y)| M_r w(y) d\mu(y) \\
      &\leq C \int_X |f(y)| M_r w(y) d\mu(y),
  \end{split}\end{equation*}

  where we have used that for any $r>1$, non-negative function $w$ with $M_r w(x)<\infty$ a.e., cube $Q_j$ and $x\in Q_j$ we have

  \begin{equation}
  \label{maximalinf}
  M_r(u\chi_{\tilde{\Omega}^c})(x)\lesssim \inf_{y\in Q_j}M_r(u\chi_{\tilde{\Omega}^c})(y)
  \end{equation}
  where the implicit constant depends only on the doubling constant $D_{\mu}$.

 This can be proved as in \cite{ACM}.  In the classical case with the Hardy--Littlewood maximal operator $M$, see for instance \cite[p. 159]{GCRF}.

  \par Combining both estimates we obtain

  \begin{equation*}
    w(\{ x\in(\wt{\O})^c: |T(g)(x)|>\lambda/2 \}) \leq \frac{C}{\lambda} (p')^p (r')^{p-1} ||f||_{L^1(M_r w)}.
  \end{equation*}

  Next we estimate the term $III$, using the estimate in Lemma \ref{GCRF_lemma} and replacing $w$ by $w\chi_{X\setminus \wt{Q}_j}$ we have

  \begin{equation*}\begin{split}
    III &= w(\{y\in X\setminus(\wt{\O})^c : |T(b)(y)|>\frac{\lambda}{2}\}) \\
        &\leq \frac{C}{\lambda} \int_{X\setminus\wt{\O}} |b(y)| w(y)d\mu(y) \\
        &\leq \frac{C}{\lambda} \sum_j \int_{X\setminus\wt{Q}_j} |b_j(y)|w(y) d\mu(y) \\
        &\leq \frac{C}{\lambda} \sum_j \int_X |b_j(y)|M(w\chi_{X\setminus\wt{Q}_j})(y)d\mu(y) \\
        &\leq \frac{C}{\lambda} \sum_j \int_{Q_j} |b(y)| M(w\chi_{X \setminus\wt{Q}_j})(y)d\mu(y) \\
        &\leq \frac{C}{\lambda} \sum_j \left(\int_{Q_j}|f(y)|Mw(y)d\mu(y) + \int_{Q_j} |g(y)| M(w\chi_{X\setminus\wt{Q}_j})(y)d\mu(y)\right)\\
        &= \frac{C}{\lambda} ( III_1 + III_2).
  \end{split}\end{equation*}

  To conclude the proof we only need to estimate $III_2$. However

  \begin{equation*}\begin{split}
    III_2 &=    \sum_j \int_{Q_j} |f_{Q_j}| M(w\chi_{\tilde{\Omega^c}})(y)d\mu(y) \\
          &\leq \sum_j \int_{Q_j} \frac{1}{\mu(Q_j)}\int_{Q_j} |f(x)| d\mu(x)  M(w\chi_{\tilde{\Omega}^c})(y)d\mu(y) \\
          &\leq \sum_j \int_{Q_j} |f(x)| d\mu(x) \inf_{Q_j} M(w\chi_{\tilde{\Omega}^c}) \\
          &\leq \sum_j \int_{Q_j} |f(x)| M(w\chi_{\tilde{\Omega}^c})(x)d\mu(x) \\
          &\leq C \int_X |f(x)| Mw(x) d\mu(x),
  \end{split}\end{equation*}

\par Combining the three estimates we have proved \eqref{weak_CLAIM} and this concludes the proof of the theorem.

\end{proof}

      \subsection{Proofs of Theorem \ref{Duo_ExtrapolationThm} and Corollary \ref{DuoThm}}\label{Subsect43}

Now, we turn to the proof of Theorem \ref{Duo_ExtrapolationThm}.  This extrapolation proof is reminiscent of the style of proof for Theorem \ref{LinearGrowthThm}.  We first need two lemmas, originally in \cite{Duo}.  The first is presented here in detail for spaces of homogeneous type.

\begin{lemma}[Factorization]\label{FactorizationLemma} If $w\in A_p$ and $u\in A_1$, we have that

\begin{enumerate}
  \item For $1\leq p < p_0 < \infty$, then $w u^{p-p_0}$ is in $A_{p_0}$ and

    \begin{equation*}\label{factorization1}
      [w u^{p-p_{0}}]_{A_{p_0}} \leq [w]_{A_p} [u]_{A_1}^{p_0-p}.
    \end{equation*}

  \item For $1<p_0<p<\infty$, then $(w^{p_0-1p}u^{p-p_0})^{\frac{1}{p-1}}$ is in $A_{p_0}$ and

   \begin{equation*}\label{factorization2}
     [(w^{p_0-1}u^{p-p_0})^{\frac{1}{p-1}}]_{A_{p_0}} \leq [w]_{A_p}^{\frac{p_0-1}{p-1}} [u]_{A_1}^{\frac{p-p_0}{p-1}}.
   \end{equation*}
\end{enumerate}
\end{lemma}

\begin{proof}[Proof of Lemma \ref{FactorizationLemma}]
  To prove $(1)$ we observe since $p-p_0<0$, using the definition of the $A_1$ constant we obtain

  \begin{equation*}\begin{split}
    \left(\avgint_{Q} u\right)^{p_0-p} u^{p-p_0} &\leq \left(\avgint_{Q} u\right)^{p_0-p} \left( \esssup_Q u^{-1} \right)^{p_0-p}\\
    &\leq [u]_{A_1}^{p_0-p}.
  \end{split}\end{equation*}

  So that

  \begin{equation}\label{fact_11}
    u^{p-p_0} \leq \left(\avgint_{Q} u\right)^{p-p_0} [u]_{A_1}^{p_0-p}.
  \end{equation}

  On the other hand, using H\"older inequality with exponents $\frac{1-p'}{1-p_0'}=\frac{p_0-1}{p-1}$ and $\frac{1-p'}{p_0'-p'}=\frac{p_0-1}{p_0-p}$, we have

  \begin{equation}\label{fact_12}\begin{split}
    &\left(\avgint_Q w^{1-p_0'} u^{(p-p_0)(1-p_0')}\right)^{p_0-1} \leq \\ &\leq \left(\avgint_Q w^{1-p'}\right)^{\frac{(p_0-1)(1-p_0')}{1-p'}} \left(\avgint_Q u^{\frac{(p-p_0)(1-p_0')(1-p')}{(p_0'-p')}}\right)^{\frac{(p_0'-p')(p_0-1)}{1-p'}} \\
    &= \left(\avgint_Q w^{1-p'}\right)^{p-1} \left(\avgint u \right)^{p_0-p}
  \end{split}\end{equation}

  Thus combining \eqref{fact_11} and \eqref{fact_12} we obtain

  \begin{equation*}
    \left(\avgint_Q w u^{p-p_0} \right) \left( \avgint_Q w^{1-p_0'} u^{(p-p_0)(1-p_0')}\right)^{p_0-1} \leq [w]_{A_p} [u]_{A_1}^{p_0-p}.
  \end{equation*}

  Finally taking supremum over all cubes $Q$ on the left hand side we get the desired result.

  \par To show $(2)$ we use a dualization argument on the weights and apply $(1)$. We have that

  \begin{equation*}\begin{split}
    \left[w^{\frac{p_0-1}{p-1}} u^{\frac{p-p_0}{p-1}}\right]_{A_{p_0}} &= \left[\left(w^{\frac{p_0-1}{p-1}} u^{\frac{p-p_0}{p-1}}\right)^{1-p_0'}\right]_{A_{p_0'}}^{\frac{1}{p_0'-1}} \\
    &= \left[w^{\frac{-1}{p-1}} u^{p'-p'_0}\right]_{A_{p_0'}}^{\frac{1}{p_0'-1}},
  \end{split}\end{equation*}

  since $\frac{p-p_0}{p-1}(1-p_0')=p'-p_0'$. Next, using $(1)$ since $w^{1-p'}\in A_{p'}$, $u\in A_1$ and $p_0'>p'$, we get the desired result

  \begin{equation*}\begin{split}
    \left[w^{\frac{p_0-1}{p-1}} u^{\frac{p-p_0}{p-1}}\right]_{A_{p_0}} &\leq [w]_{A_p}^{\frac{p'-1}{p_0'-1}} [u]_{A_1}^{\frac{p_0'-p'}{p_0'-1}} \\
    &\leq [w]_{A_p}^{\frac{p_0-1}{p-1}} [u]_{A_1}^{\frac{p-p_0}{p-1}},
  \end{split}\end{equation*}

  where we have used in the next to last inequality that $[w]_{A_p}^{p'-1}=[w^{1-p'}]_{A_{p'}}$ and in the last inequality the fact that $\frac{p_0'-p'}{p_0'-1}=\frac{p-p_0}{p-1}$ and $\frac{p'-1}{p_0'-1}=\frac{p_0-1}{p-1}$.

\end{proof}

The second lemma is a Rubio de Francia algorithm for building $A_1$ weights. We omit the proof of this result since it follows exactly in the same manner as in \cite{Duo}.

\begin{lemma}[Rubio de Francia's algorithm]\label{Duo_RdF_algorithm}
  Let $p>1$. Let $f$ be a nonnegative function in $L^p(w)$ and $w\in A_p$. Let $M^k$ be the $k^{th}$ iterate of $M$, $M^0=f$, and $||M||_{L^p(w)}$ be the norm of  $M$ as a bounded operator on $L^p(w)$. Define

  \begin{equation*}
    Rf(x)=\sum_{k=0}^{\infty} \frac{M^k f(x)}{2||M||_{L^p(w)}^k}.
  \end{equation*}
  Then $f(x)\leq Rf(x)$ a.e., $||Rf||_{L^p(w)} \leq 2 ||f||_{L^p(w)}$, and $Rf$ is an $A_1$ weight with constant $[Rf]_{A_1}\leq 2 ||M||_{L^p(w)}$.

\end{lemma}

\begin{proof}[Proof of Theorem \ref{Duo_ExtrapolationThm}]
  The proof of this result follows from Duoandikoetxea's proof except for the fact that we have to replace the sharp bound for the Hardy--Littlewood maximal function by the new Buckley's theorem in Hyt\"onen and Kairema \cite{HK}, so the constant $C_2$ in the proof now depends on $p$ and $X$.
\end{proof}

\begin{proof}[Proof of Theorem \ref{DuoThm}]
  It follows directly from the proof in \cite{Duo}.
\end{proof}

      \subsection{Proof of Theorem \ref{mixedsharp} and Corollary \ref{commutator_thm}}\label{Subsect44}

To prove the main result in this section we first need the following lemmas.  The first is a sharp reverse H\"older inequality for $A_{\infty}$ weights in spaces of homogeneous type adapted from an argument due to Hyt\"onen, P\'erez, and Rela whose proof can be found in \cite{HPR}.  Before we state and prove this, we note that in this same paper there is a weak version of this inequality stated below.  They call this result a weak inequality since on the right hand side we have the dilation $2\kappa B$ of the ball $B$.

\begin{lemma}[\cite{HPR}]
Let $w\in A_{\infty}$ and define \[r = r_w = 1+\frac{1}{\tau[w]_{A_{\infty}}} = 1+\frac{1}{6(32\kappa^2(4\kappa^2+\kappa)^2)^{D_{\mu}}[w]_{A_{\infty}}}\] where $\tau$ depends on $\kappa$, the quasimetric constant of $X$.  Then \[\left( \avgint_Bw^rd\mu\right) ^{1/r}\leq \left( 2(4\kappa)^{D_{\mu}}\avgint_{2\kappa B}wd\mu\right),\] for any ball $B\in X$.
\end{lemma}

However, this lemma is not sufficient for our purposes.  The difficulty lies in the fact that the Fujii-Wilson $A_{\infty}$ constant is comparable when it is defined with respect to cubes or balls, but the constant of comparison depends on the weight $w$.  This provides a difficulty in converting between the constants, and since cubes are essential in the following lemmas, we need a sharp reverse H\"older inequality with cubes.  Here is the lemma that we use with respect to cubes.

\begin{lemma}\label{SRHI}
Let $w\in A_{\infty}$ and let
\[0<r \leq\frac{1}{\tau[w]_{A_{\infty}}-1} =\frac{1}{2D[w]_{A_{\infty}}-1},\]
with $D = 1/\ve$, where $\ve$ is the absolute constant appearing in the dyadic decomposition of $X$.  Then
\[\avgint_Q w^{1+r} d\mu\leq  2 \left( \avgint_{Q}wd\mu\right)^{1+r},\]
for any cube $Q\subset X$.
\end{lemma}

The proof will use the following sublemma.

\begin{lemma}\label{SRHsublemma}
Let $w\in A_{\infty}$ and $Q_0$ a cube.  Then for all
\[0<r \leq\frac{1}{2D[w]_{A_{\infty}}-1}\]
we have
\[\avgint_Q(Mw)^{1+r}d\mu\leq  2[w]_{\infty}\left( \avgint_{Q}wd\mu\right)^{1+r}.\]
\end{lemma}

\begin{proof}[Proof of Lemma \ref{SRHsublemma}]
Assume without loss of generality that $w = w\chi_{Q_0}$.  Let $\Omega_{\lambda} = Q_o\cap \{Mw>\lambda\}$.  Then
\[\int_{Q_0}(Mw)^{1+r} d\mu = \int_0^{\infty}r\lambda^{r-1}Mw(\Omega_{\lambda})d\mu d\lambda\]\[=\int_0^{w_{Q_0}}r\lambda^{r-1}\int_{Q_0}Mwd\lambda + \int_{w_{Q_0}}^{\infty}r\lambda^{r-1}Mw(\Omega_{\lambda})d\lambda.\]
Now select a dyadic cube $Q_j$ if it is maximal with respect to the following condition: $\lambda < w_{Q_j}$.  Then $\Omega_{\lambda} = \cup_jQ_j$ where  $\lambda <\avgint_{Q_j}w \leq \frac{1}{\ve}\lambda$ and $\ve$ is the absolute constant from Theorem \ref{dyadic}.  Hence we have
\[\int_{Q_0}(Mw)^{1+r}\leq w_{Q_0}^r[w]_{\infty}w(Q_0)+ \int_{w_{Q_0}}^{\infty}r\lambda^{r-1}\sum_j\int_{Q_j}Mwd\mu d\lambda.\]  Now we can localize
\[Mw(x) = M(w\chi_{Q_j})(x)\]
by the maximality of the $Q_j's$ for any $x\in Q_j$.  Then,

\begin{equation*}\begin{split}
  \int_{Q_j}Mw d\mu &= \int_{Q_j}M(w\chi_{Q_j}) \leq [w]_{\infty}w(Q_j)\leq [w]_{\infty}w(\hat{Q_j}) \\
  &=[w]_{\infty}w_{\hat{Q_j}}\mu(\hat{Q_j}) \leq [w]_{\infty}\lambda\frac{1}{\ve}\mu(Q_j),
\end{split}\end{equation*}

where $\hat{Q_j}$ is the parent of the cube $Q_j$ and we have used the definition of $A_{\infty}$ and the maximality and containment properties of the cubes.  Call $\frac{1}{\ve} = D$.  Hence

\[\sum_j\int_{Q_j}Mw d\mu \leq \sum_j[w]_{\infty}\lambda D \mu(Q_j) \leq [w]_{\infty}\lambda D\mu(\Omega_{\lambda})\]

so

\[\avgint_{Q_0}(Mw)^{1+r} \leq w_{Q_0}^{r}[w]_{\infty}w(Q_0) + r[w]_{\infty}D\int_{w_{Q_0}}^{\infty}\lambda^r\mu(\Omega_{\lambda})d\lambda.\]

Dividing by $w(Q_0)$, we obtain

\[\avgint_{Q_0}(Mw)^{1+r}\leq w_{Q_0}^{1+r}[w]_{\infty}+\frac{rD[w]_{\infty}}{1+r}\avgint_{{Q_0}} (Mw)^{1+r},\]

so by subtracting the last term on the left hand side from both sides of the equation, so to get the desired constant of 2 we must have that \[1-\frac{rD[w]_{\infty}}{1+r}\geq \frac{1}{2},\] which after some calculation results in choosing $0<r\leq \frac{1}{2D[w]_{\infty}-1}$ as stated.
\end{proof}

\begin{proof}[Proof of Lemma \ref{SRHI}]
Without loss of generality let $w = w\chi_{Q_0}$.  Then

\begin{equation*}
  \int_{Q_0}w^{1+r} \leq \int_{Q_0}(Mw)^rw = \int_0^{\infty}r\lambda^{r-1}w(\Omega_{\lambda})d \lambda,
\end{equation*}
where $\Omega_{\lambda} = Q_0\cap \{Mw>\lambda\}$.  Note that as in the previous lemma we can decompose $\Omega_{\lambda} = \cup_j Q_j$ where the $Q_j$ are the Calder\'on-Zygumnd cubes.  Then splitting up the integral we get

\begin{equation*}\begin{split}
  &\int_0^{w_{Q_0}}r\lambda ^{r-1}d\lambda + \int_{w_{Q_0}}^{\infty}r\lambda ^{r-1}w(\Omega_{\lambda})d\lambda \\
  &\leq w_{Q_{0}}^rw(Q_0)+ \int_{w_{Q_0}}^{\infty} r\lambda^{r-1}\sum_jw(Q_j)d\lambda.
\end{split}\end{equation*}

Now by the decomposition, we have that $w_{Q_j}\leq C_X\lambda \mu(Q_j)$, where $C_X = \frac{1}{\ve}$ since the decomposition is with respect to dyadic cubes, so we get

\[\int_{Q_0}(Mw)^rwd\mu \leq w_{Q_0}^rw(Q_0)+rC_X\int_{w_{Q_0}}^{\infty}r\lambda^r\sum_j\mu(Q_j)d\lambda\]\[\leq w_{Q_0}^rw(Q_0)+rC_X\int_{w_{Q_0}}^{\infty}\lambda^r\mu(\Omega_{\lambda})d\lambda \leq w_{Q_0}^rw(Q_0)+\frac{rC_X}{1+r}\int_{Q_0}(Mw)^{1+r}.\]

Hence, dividing by $w(Q_0)$ and using Lemma \ref{SRHsublemma}, we arrive at

\begin{equation*}\begin{split}
  \avgint_{Q_0}w^{1+r} &\leq w_{Q_0}^{1+r}+ \frac{r C_X)2[w]_{\infty}}{1+r}\left( \avgint_{Q_0}w\right)^{1+r} \\
                       &\leq \frac{r C_X 2[w]_{\infty}+1+r}{1+r}\left( \avgint_{Q_0}w\right)^{1+r}.
\end{split}\end{equation*}  Therefore, choosing $\ve$ in the mentioned range, we can
make the constant on the right hand side less than or equal to 2.

\end{proof}

The next lemma is a precise version of John-Nirenberg inequality in spaces of homogeneous type that will be very useful in the following results.

\begin{lemma}[John-Nirenberg inequality]\label{JohnNirenberg_lemma} There are absolute constants $0\leq \alpha_X < 1 < \beta_X$ such that

\begin{equation}\label{JohnNirenberg_ineq}
  \sup_Q \frac{1}{\mu(Q)} \int_Q \exp{\frac{\alpha_X}{||b||_{BMO}}|b(y)-b_Q|}d\mu(y) \leq \beta_X.
\end{equation}

In fact, we can take $\alpha_X=\ln{\sqrt[3]{2^{\ve}}}$, where $0<\ve<1$ is an absolute constant.
\end{lemma}

\begin{proof}[Proof of Lemma \ref{JohnNirenberg_lemma}]
For the proof of this lemma we follow the scheme of proof showed in \cite{Jour}. Suppose that $b$ is bounded, so that the above supremum makes sense for all $\alpha$. Then we will prove \eqref{JohnNirenberg_ineq} with a bound independent of $||b||_{\infty}$.

\par Fix a cube $Q_0$ and a dyadic cube $Q\in\mathscr{D}(Q_0)$. Denote by $\tilde{Q}$ the parent of $Q$, namely, the unique element in $\mathscr{D}(Q_0)$ which contains $Q$ and lies in the previous generation of cubes.

\par Then we can show that

\begin{equation}\label{lemma_JN}
  \left|b_Q-b_{\tilde{Q}}\right| \leq \frac{1}{\ve} ||b||_{BMO},
\end{equation}

where $0<\ve<1$ is an absolute constant as in \cite{AV} (see also \cite{C} or \cite{HK} for further details). Indeed,

\begin{equation*}\begin{split}
  \left| b_Q - b_{\tilde{Q}} \right| &\leq \frac{1}{\mu(Q)}\int_Q \left| b - b_{\tilde{Q}}\right| \\
  &\leq \frac{\mu(\tilde{Q})}{\mu(Q)} \frac{1}{\mu(\tilde{Q})} \int_{\tilde{Q}} \left|b - b_{\tilde{Q}} \right| \\
  &\leq \frac{1}{\ve} ||b||_{BMO}.
\end{split}\end{equation*}

Next, consider the Calder\'on--Zygmund decomposition of $\left(b-b_{Q_0}\right)\chi_{Q_0}$ described in \cite[Thm. 2.7.]{ACM} for the level $2||b||_{BMO}$. Then there exists a collection of pairwise disjoint cubes $\{Q_i\}\subset \mathscr{D}$, maximal with respect to inclusion, satisfying

\begin{equation*}
  2||b||_{BMO} < \frac{1}{\mu(Q_i)}\int_{Q_i} \left| \left(b - b_{Q_0}\right)\chi_{Q_0} \right| < 2 C_X ||b||_{BMO}
\end{equation*}

and

\begin{equation*}
  \left| \left(b - b_{Q_0}\right)\chi_{Q_0} \right| < 2||b||_{BMO}, \text{  on $(\cup Q_i)^c$.}
\end{equation*}

Clearly, $Q_i\subset Q_0$ for each $j$, and

\begin{equation*}
  \mu(\cup Q_i) \leq \frac{||(b-b_{Q_0})\chi_{Q_0}||_{L^1}}{2||b||_{BMO}} \leq \frac{\mu(Q_0)}{2}.
\end{equation*}

Since the cubes $Q_i$ are maximal, we have that $(|b-b_{Q_0}|)_{\widetilde{Q_i}}\leq 2||b||_{BMO}$. Next, using the last inequality together with \eqref{lemma_JN} we get

\begin{equation*}
  |b_{Q_i}-b_{Q_0}| \leq |b_{Q_i}-b_{\widetilde{Q_i}}| + |b_{\widetilde{Q_i}}-b_{Q_0}| \leq \left(\frac{1}{\ve}+2\right)||b||_{BMO}.
\end{equation*}

 Denote $X(\alpha)=\sup_Q \frac{1}{\mu(Q)}\int_Q \exp{\frac{\alpha}{||b||_{BMO}}|b-b_Q|}d\mu(x)$, which is finite since we are assuming that $b$ is bounded. From the properties of the cubes $Q_i$ we arrive at

\begin{equation*}\begin{split}
  &\frac{1}{\mu(Q_0)}\int_{Q_0} \exp{\left(\frac{\alpha}{||b||_{BMO}}|b-b_{Q_0}|\right)}d\mu(x) \\
  &\leq \frac{1}{\mu(Q_0)}\int_{Q_0\setminus \cup Q_i} e^{2\alpha}d\mu(x) \\ &+ \sum_{j}\frac{\mu(Q_i)}{\mu(Q_0)}\frac{1}{\mu(Q_i)}\left(\int_{Q_i}\exp{\left(\frac{\alpha}{||b||_{BMO}} |b-b_{Q_i}|d\mu(x) \right)}e^{(\frac{1}{\ve}+2)\alpha}\right)\\
  &\leq e^{2\alpha} +\frac{1}{2} e^{(\frac{1}{\ve}+2)\alpha} X(\alpha).
\end{split}\end{equation*}

Taking the supremum over all cubes $Q_0$, we get the bound

\begin{equation*}
  X(\alpha) \left( 1 - \frac{1}{2} e^{(\frac{1}{\ve}+2)\alpha} \right) \leq  e^{2\alpha},
\end{equation*}

which implies that $X(\alpha)\leq C$, if $\alpha$ is small enough.

\par Since $0<\ve<1$, if we impose that $\frac{1}{2}e^{(\frac{1}{\ve}+2)\alpha}<1$, then $\alpha < \frac{\ve \ln{2}}{2\ve+1}$. Therefore we can choose an smaller parameter $\alpha$, such as $\alpha_X=\ln{\sqrt[3]{2^{\ve}}}$.
\end{proof}

Now we will prove two lemmas related to the $A_2$ and $A_{\infty}$ constants of a particular weight that we will need in the following, extended from those in \cite{HPR}.

\begin{lemma}\label{JNlemma}
There are absolute constants $\gamma$ and $c$ such that
\[  [we^{2Rez b}]_{A_2} \leq c[w]_{A_2}\]
for all
\[|z|\leq \frac{\gamma}{\|b\|_{BMO}([w]_{A_{\infty}}+[\sigma]_{A_{\infty}})},\]
where $\gamma=\max\{C_1\alpha_X, C_2\alpha_X\}$ with $C_1$ and $C_2$ absolute constants.
\end{lemma}

\begin{proof}[Proof of Lemma \ref{JNlemma}]
We will use the sharp reverse H\"older inequality twice, first for $r= 1+\frac{1}{\tau[w]_{A_{\infty}}}$ and then for $r=1+\frac{1}{\tau[\sigma]_{A_{\infty}}}$.  With the sharp reverse H\"older inequality for the first choice of $r$, H\"older's inequality and the sharp John-Nirenberg inequality \eqref{JohnNirenberg_ineq}, we have

\[\avgint_Qwe^{Rez b} \leq \left( \avgint_Qw^r\right) ^{1/r}\left( \avgint_Qe^{r'Rez(b-b_Q)}\right) ^{1/r'} e^{Rez b_Q}\]\[\leq \left( 2\avgint_{Q}w\right) \cdot \beta_X \cdot e^{Rez b_Q},\]

for $|z| \leq \frac{C_1 \alpha_X}{\|b\|_{BMO}[w]_{A_{\infty}}}.$  Note that the constant $\alpha_X$ comes from \eqref{JohnNirenberg_ineq} and $C_1$ is an absolute constant from the sharp reverse H\"older inequality since by our choice of $r$, $r' = C_1[w]_{A_{\infty}}$ (we can even calculate that $\tau [w]_{A_{\infty}} <r' \leq (\tau +1) [w]_{A_{\infty}}$.  We can also get a similar bound as above for the second choice of $r =  1+\frac{1}{\tau[\sigma]_{A_{\infty}}}$, giving us
\[  \avgint_Qw^{-1}e^{-Rez b}\leq \left( 2\avgint_{Q}w^{-1}\right) \cdot \beta_X\cdot e^{-Rez b_Q} \] for $|z| \leq \frac{C_2 \alpha_X}{\|b\|_{BMO}[\sigma]_{A_{\infty}}}.$
Multiplying these two estimates and taking supremum, we finish the proof by showing that for all $z$ as in the assumption
\[\left( \avgint_Qwe^{Rez  b}\right) \left( \avgint_Qw^{-1}e^{-Rez b}\right) \leq 4\beta_X^2[w]_{A_2}.\]
\end{proof}

We also have a similar lemma for the $A_{\infty}$ weight constant.

\begin{lemma}\label{JNlemma_Ainfy}
There are absolute constants $\gamma$ and $c$ such that
\[  [we^{2Rez b}]_{A_{\infty}} \leq c[w]_{A_{\infty}}\]
for all
\[|z|\leq \frac{\gamma'}{\|b\|_{BMO}([w]_{A_{\infty}})},\]
where we can take \[\gamma'=\frac{\alpha_X}{4\tau}\]  being $\tau$ an absolute constant from Lemma \ref{SRHI}.
\end{lemma}

\begin{proof}[Proof of Lemma \ref{JNlemma_Ainfy}]
The proof follows in a similar way as in \cite{CPP}, substituting the appropriate constants from the sharp John Nirenberg inequality in Lemma \ref{JohnNirenberg_ineq} and the sharp reverse H\"older inequality in Lemma \ref{SRHI}.
\end{proof}

Next we will prove Theorem \ref{mixedsharp} where a mixed $A_2-A_{\infty}$ bound for Calder\'on--Zygmund operators in spaces of homogeneous type is obtained.  Due to Lerner's decomposition in spaces of homogeneous type from \cite{ACM}, we can fairly easily prove the mixed result.  The proof essentially follows from \cite{Hyt2}.  Only a brief sketch is given below.

\begin{proof}[Proof of Theorem \ref{mixedsharp}]
As stated in \cite[Sect. 2D]{Hyt2}, Theorem \ref{mixedsharp} follows from verifying the following testing conditions:
\begin{enumerate}
\item{$\|S_Q(\sigma\cdot \chi_Q)\|_{L^2(w)}\leq C_1\|\chi_Q\|_{L^2(\sigma)}$}
\item{$\|S_Q(w\cdot \chi_Q)\|_{L^2(\sigma)}\leq C_2\|\chi_Q\|_{L^2(w)}$}
\end{enumerate}
where \[S_Qf = \sum_{L\in S, L\subseteq Q}\left( \avgint_L f\right)\chi_L,\] and $S$ is a sparse family (this is a sparse operator).
This is still the case in spaces of homogeneous type due to the mentioned result of Treil \cite{T}, which holds in spaces of homogeneous type.

To verify the testing conditions, one simply follows the argument outlined in \cite[Sect. 5A]{Hyt2}.
\end{proof}

Finally, we can prove the results concerning to the commutator and its iterates.

\begin{proof}[Proof of Corollary \ref{commutator_thm}]

  Firstly, we start proving \eqref{comm1}.  Let us ``conjugate'' the operator $T$ as follows, that is, for any complex number $z$ we define

  \begin{equation}\label{Tz}
    T_z(f) = e^{zb} T(e^{-zb}f).
  \end{equation}

  By using the Cauchy integral theorem, we get for appropriate functions,

  \begin{equation*}
    T_b(f)=\frac{d}{dz}T_z(f)|_{z=0} = \frac{1}{2\pi i} \int_{|z|=\ve} \frac{T_z(f)}{z^2}d\mu(z), \text{\quad} \ve>0,
  \end{equation*}
  Therefore we can write

  \begin{equation*}\begin{split}
  \|T_b(f)\|_{L^2(w)} &= \left\| (2\pi i)^{-1}\int_{|z|=\ve}\frac{T(fe^{-zb})}{z^2} e^{zb}\right\|_{L^2(w)} \\
   &\leq \frac{1}{C\ve^2}\int_{|z|=\ve}\left( \int_X|T(fe^{-zb})e^{zb}|^2wd\mu(x)\right) ^{1/2}|d\mu(z)|\\
    &=\frac{C}{\ve}\|T(fe^{-zb})\|_{L^2(we^{2Rez b})} \\
    &\leq \frac{C}{\ve}[we^{2Rez b}]_{A_2}^{1/2}([we^{2Rezb}]_{A_{\infty}}+[\sigma e^{2Rezb}]_{A_{\infty}})^{1/2} \\
    &\times \left( \int_X|fe^{-zb}|^2we^{2Rezb}d\mu\right) ^{1/2} \\
    &= \frac{C}{\ve}[we^{2Rez  b}]_{A_2}^{1/2}([we^{2Rez b}]_{A_{\infty}}+[\sigma e^{2Rez b}]_{A_{\infty}})^{1/2} \\
    &\times \left( \int_X|f|^2wd\mu(x) \right) ^{1/2},
  \end{split}\end{equation*}
  where we have used the Minkowski inequality for integrals and the $A_2$ theorem for spaces of homogeneous type \cite{AV}.
  We also have
  \begin{equation}
  \label{weightexponential}
  [we^{2Rez b}]_{A_2} \leq c[w]_{A_2}
  \end{equation}
   and similarly for the $A_{\infty}$ constants,
   for all $|z|\leq \frac{\delta}{\|b\|_{BMO}([w]_{A_{\infty}}+[\sigma]_{A_{\infty}})}$ where the $\delta$ is the minimum of the absolute constants from the corresponding lemmas.

  All that remains is to bound

  \[\frac{C}{\ve}[w]_{A_2}\|f\|_{L^2(w)}.\]

 Since $|z| = \ve$ we are restricted to certain $\ve$ by \eqref{weightexponential}, so we choose $\ve = \frac{\gamma}{\|b\|_{BMO}([w]_{A_{\infty}}+[\sigma]_{A_{\infty}})}$, so that \[\frac{1}{\ve} = \frac{1}{\delta}([w]_{A_{\infty}}+[\sigma]_{A_{\infty}})\|b\|_{BMO}\] as wanted.

 Putting everything together gives us the desired bound

 \begin{equation*}
    ||T_b(f)||_{L^2(w)} \leq C [w]_{A_2}^{1/2} ([w]_{A_{\infty}}+[\sigma]_{A_{\infty}})^{3/2} \|b\|_{BMO} ||f||_{L^2(w)}.
 \end{equation*}

 Finally, to prove the general estimate \eqref{comm2}, we use again the Cauchy integral theorem to write the k-th commutator for appropriate functions as

  \begin{equation*}
    T^k_b(f)= \frac{d^k}{dz^k}T_z(f)|_{z=0} = \frac{k!}{2\pi i} \int_{|z|=\ve} \frac{T_z(f)}{z^{k+1}}d\mu(z), \text{\quad} \ve>0,
  \end{equation*}

  where $T_z$ is defined as in \eqref{Tz}. Then, following the computation for $T_b$ we can arrive at the desired bound for $T_b^k$.

\end{proof}

\begin{remark}
  Corollary \ref{commutator_thm} can be proved under the weaker assumption that $T$ is a linear operator that satisfies the sharp weak mixed $A_2-A_{\infty}$ in spaces of homogeneous type.
\end{remark}

%
%

\thanks{\textbf{Acknowledgments.}{The authors would like to express their gratitude to Prof. Carlos P\'erez for helpful comments and motivating discussions in Seville, Spain.  The first author would like to thank Tuomas Hyt\"onen for suggestions and the opportunity to enjoy the fantastic working environment at the University of Helsinki.  The second author would also like to thank Prof. Jill Pipher for the invitation to continue this collaboration with the first author at Brown University, Providence.}

%
%
%

\end{document}